\documentclass[12pt,a4paper,reqno]{amsart}

\usepackage{amssymb}
\usepackage{enumerate}
\usepackage{xcolor}

\usepackage[shortlabels]{enumitem}
\setlist[1]{wide}
\setlist[2]{leftmargin=15mm}
\setlist[enumerate]{label=\rm{(\arabic*)}}
\setlist[enumerate,2]{label=\rm({\it\roman*}), }
\setlist[itemize]{label=\raisebox{0.25ex}{\tiny$\bullet$}}

\usepackage{thmtools}
\usepackage{thm-restate}
\usepackage[backref, colorlinks, linktocpage, citecolor = blue, linkcolor = blue]{hyperref}

\usepackage{cleveref}

\newtheorem{maintheorem}{Theorem}

\newtheorem{mainprop}[maintheorem]{Proposition}

\newtheorem{theorem}{Theorem}[section]
\newtheorem{proposition}[theorem]{Proposition}

\newtheorem{lemma}[theorem]{Lemma}

\theoremstyle{definition}

\newtheorem{remark}[theorem]{Remark}

\newtheorem{example}[theorem]{Example}

\newcommand{\suchthat}{\;\ifnum\currentgrouptype=16 \middle\fi|\;}

\renewcommand{\k}{\mathrm{k}}
\newcommand{\Aut}{\mathrm{Aut}}
\newcommand{\Tame}{\mathrm{Tame}}
\newcommand{\Tra}{\mathrm{Tra}}
\newcommand{\Tri}{\mathrm{Tri}}

\newcommand{\End}{\mathrm{End}}

\newcommand{\GL}{\mathrm{GL}}
\newcommand{\SL}{\mathrm{SL}}
\newcommand{\SAut}{\mathrm{SAut}}

\newcommand{\F}{\mathbb{F}}

\newcommand{\SAff}{\mathrm{SAff}}

\newcommand{\Aff}{\mathrm{Aff}}
\newcommand{\A}{\mathbb{A}}
\newcommand{\Jac}{\mathrm{Jac}}

\usepackage[all]{xy}
\xyoption{matrix}

\title[$\SAut(\A^2)$ is topologically simple]{Topological simplicity of the group of automorphisms of the affine plane}
\author{J\'er\'emy Blanc}
\address{J\'er\'emy Blanc\\ Institut de math\'ematiques \\ 
Universit\'e de Neuch\^atel\\
Rue Emile-Argand 11\\
CH-2000 Neuch\^atel \\
Switzerland}
\email{jeremy.blanc@unine.ch}\thanks{The author gratefully acknowledges support by the Swiss National Science Foundation Grant  ``Birational transformations of higher dimensional varieties'' 200020\_214999}
\date{\today}

\begin{document}
\begin{abstract}
We prove that the group $\mathrm{SAut}_{\mathrm{k}}(\mathbb{A}^2)$ is simple as an algebraic group of infinite dimension, over any infinite field $\mathrm{k}$, by proving that any closed normal subgroup is either trivial or the whole group. In higher dimension, we show that closed normal subgroups contain all tame automorphisms.

The case of finite fields, very different, is also discussed.
\end{abstract}
\maketitle
\tableofcontents
\section{Introduction}
Let $\k$ be a field and $n\ge 1$ be an integer. The group $\Aut_\k(\A^n)$ of polynomial automorphisms of $\A^n$ (see Section~\ref{NotRem} for precise definitions and reminders) is, despite much work, highly not understood for $n\ge 3$. For $n=3$ and $\mathrm{char}(\k)=0$, the group is known to be not generated by affine and triangular automorphisms: there are \emph{non-tame automorphisms} \cite{ShestakovUmirbaev}, like for instance the famous Nagata automorphism of \cite[Section 2.1]{Nagata} (see Example~\ref{NagataAuto}). We however do not have a set of explicit generators for any $n\ge 3$ and any field $\k$. The group $\Aut_\k(\A^2)$ is much better understood, for any field $\k$. It has a well-known structure of an amalgamated product, which is very useful for a lot of questions on the group: one knows its finite subgroups, algebraic subgroups, the possible dynamical degrees, ...

The group  $\Aut_\k(\A^n)$ is contained in $\End_\k(\A^n)\simeq \k[x_1,\ldots,x_n]^n$, on which one can put two natural topologies (see \cite{Stampfli}), one being weaker than the other one. The group $\Aut_\k(\A^n)$ then has the structure of an \emph{infinite dimensional algebraic group}, or \emph{ind-group}, as explained in \cite{Sha81} (see also~\cite{Kambayashi,Kumar,FurterKraft}). One can then ask for the structure of the group, and in particular ask for possible closed normal subgroups. Note that this only makes sense  when $\k$ is infinite, as otherwise $\Aut_\k(\A^n)$ is a discrete groupe. The Jacobian gives a natural surjective group homomorphism
 \[\Jac\colon \Aut_\k(\A^n)\twoheadrightarrow \mathbb{G}_M(\k)\simeq \k^*\]
 whose kernel is the closed (in any of the two topologies) normal subgroup $\SAut_\k(\A^n)$. One can then show that any non-trivial closed normal subgroups of $\Aut_\k(\A^n)$ contains all tame automorphisms \cite[Proposition 15.11.5]{FurterKraft} (see  Proposition~\ref{NAut} below for a generalisation).

The main question in this context, still open in general, is whether $\SAut_\k(\A^n)$ is simple as an ind-group. In \cite[Theorem 5]{Sha81}, it is claimed that $\SAut_\k(\A^n)$ is simple as an ind-group, but the proof uses the fact that the associated Lie algebra is simple, and thus also uses the Lie correspondence established in \cite[Theorem 1]{Sha81}, which is know to be false (see \cite{FurterKraft}). The arguments used in \cite[Proposition 15.11.5]{FurterKraft} dot not work here, as the conjugation by an homothety, not contained in $\SAut_\k(\A^n)$, is needed. The simplicity of $\SAut_\k(\A^n)$ was then open until now, for each $n\ge 2$. We prove it for $n=2$ and give a first step for $n\ge 3$:

\begin{maintheorem}\label{TheoremSimple}
Let $\k$ be an infinite field and let $n\ge 1$ be an integer. If $N\subseteq \SAut_\k(\A^n)$ is a closed normal subgroup, then either $N=\{\mathrm{id}\}$ or $N$ contains any tame automorphism of $\SAut_\k(\A^n)$. In particular, if $n=2$, then $\SAut_\k(\A^2)$ is simple as an ind-group.
\end{maintheorem}
In Theorem~\ref{TheoremSimple}, the closedness is asked in any of the two topologies defined before. If $\k$ is a local field, one can also put a topology on $\SAut_\k(\A^n)$ using the topology of the field, and the result still will work. The construction actually proves that the results hold for any topology such that for each $f\in \SAut_{\k[t]}(\A^n)$ and each closed subgroup $N\subseteq \SAut_\k(\A^n)$, if replacing $t$ with any element of $\k^*$ gives an element of $N$, then replacing by $t$ by zero also gives an element of $N$.

Without the topology, the group $\SAut_\k(\A^2)$ is known to be not simple since a long time (this is the main result of \cite{Danilov}, see also \cite{FurterLamy} for a modern and more detailed treatment). The (non-)simplicity of the group $\SAut_\k(\A^n)$ as an abstract group  is however open for each $n\ge 3$ and each infinite field $\k$ (for finite fields, it suffices to look at the actions on the set of $\k$-rational points and the kernel will give normal subgroups).

In  \cite{Blanc}, it was actually proven that the conjugacy class of a general element of $\SAut_\k(\A^2)$ was closed, contrary to the case of $\Aut_\k(\A^2)$: for most elements $f\in \SAut_\k(\A^2)$, the set $\{af a^{-1}\mid a\in \SAut_\k(\A^2)\}$ is closed. This was considered as a first step towards a proof of the non-simplicity of $\SAut_\k(\A^2)$ as an algebraic group. The next step actually does not work. Writing \[\Tra_n(\k)= \{(x_1+v_1,\ldots,x_n+v_n)\mid (v_1,\ldots,v_n)\in \k^n\}\] the group of translations, we prove that for each non-trivial $f\in \SAut_\k(\A^n)$, there is a non-trivial translation in the closure of \[\{\rho (\tau^{-1} f^{-1} \tau f) \rho^{-1}\mid \rho,\tau \in \Tra_n(\k)\}.\]
This implies that a non-trivial closed normal subgroup of $\SAut_\k(\A^n)$ contains all translations:

\begin{mainprop}\label{PropLimitTrans} Let $\k$ be an infinite field, let $n\ge 2$ be an integer. Each closed normal subgroup $N\subseteq \SAut_\k(\A^n)$ with $N\not=\{\mathrm{id}\}$ contains all translations.
\end{mainprop}

It then remains to see that any normal subgroup of $\SAut_\k(\A^n)$ containing all translations contains all tame automorphisms of Jacobian~$1$. This is true for each infinite field and was proven by Drew Lewis in \cite[Theorems 6 and 9]{Lewis}; we include the proof in Section~\ref{NormalTransLin}, as it is short and nice, for self-containedness. Theorem~\ref{TheoremSimple} follows from Proposition~\ref{PropLimitTrans} and Proposition~\ref{PropLewis} below.

\begin{mainprop} \cite[Theorems 6 and 9]{Lewis} \label{PropLewis}
Let $\k$ be an infinite field and let $n\ge 2$ be an integer. 

\begin{enumerate}
\item
Any normal subgroup $N\subseteq \SAut_\k(\A^n)$ containing all translations contains also $\SL_n(\k)$.
\item
Any normal subgroup $N\subseteq \SAut_\k(\A^n)$ containing  $\SL_n(\k)$ contains all tame automorphisms of Jacobian $1$.
\end{enumerate}
\end{mainprop}
Note that if $\k$ is finite, then most of the results become false. As said before, we cannot really work with the topology here, so  for finite fields $\k$ we consider the group as a discrete one, and obtain the following results in dimension $2$, that is opposite to Proposition~\ref{PropLewis}:
\begin{mainprop}\label{PropFiniteA2}
For each finite field $\k$, there is a surjective group homomorphism 
\[\SAut_\k(\A^2)\twoheadrightarrow (\k[x],+)\]
whose kernel contains all translations and all elements of $\SL_2(\k)$. 

In particular, $\SAut_\k(\A^2)$ is not a perfect group, is not finitely generated, and the normal subgroup generated by $\SAff_\k(\A^2)$ is a strict subgroup of $\SAut_\k(\A^2)$.

There exists moreover a a surjective group homomorphism 
\[\SAut_{\F_2}(\A^2)\twoheadrightarrow (\F_2[x],+)\]
whose kernel contains all translations but does not contain $\SL_2(\F_2)$.
\end{mainprop}

We finish this introduction by giving a result on $\Aut_\k(\A^n)$. 
\begin{mainprop}\label{NAut}
Let $\k$ be an infinite field, let $n\ge 2$. If $N\subseteq \Aut_\k(\A^n)$ is a normal subgroup, with $\Jac(N)=\k^*$, then $N=\Aut_\k(\A^n)$.\end{mainprop}

\section{Notation and reminders}\label{NotRem}
If $R$ is an integral domain and $n\ge 1$ is an integer, we denote by $\End_{R}(\A^n)$ the set of \emph{$R$-endomorphisms} of the affine space $\A^n$. Each element $f\in\End_{R}(\A^n)$ can be written as $f=(f_1,\ldots,f_n)$, where $f_1,\ldots,f_n\in R[x_1,\ldots,x_n]$ are polynomials in $n$ variables, and corresponds to the endomorphism 
\[(x_1,\ldots,x_n)\mapsto (f_1(x_1,\ldots,x_n),\ldots,f_n(x_1,\ldots,x_n))\] of $\A^n$. It also corresponds to an element $f^*\in \End_R(R[x_1,\ldots,x_n])$ that sends $P$ onto $f\circ P=P(f_1,\ldots,f_n)$. 

The \emph{derivative} of $f=(f_1,\ldots,f_n)$ is the matrix 
\[D(f)=Df=(\frac{\partial f_i}{\partial x_j})_{i,j=1}^n\in \mathrm{Mat}_{n\times n}(\k[x_1,\ldots,x_n]).\]
and its determinant is the \emph{Jacobian} of $f$, that is
\[\Jac(f)=\det(Df)\in \k[x_1,\ldots,x_n].\]
We can compose endomorphisms; the group of invertible elements is denoted by $\Aut_{R}(\A^n)$ and consists of \emph{$R$-automorphisms} of the affine space $\A^n$.
As $D(f\circ g)=g^*(Df)\cdot Dg$ for all $f,g\in \End_{R}(\A^n)$, one finds that $\Jac(f)\in (R[x_1,\ldots,x_n])^*=R^*$ for each $f\in \Aut_R(\A^n)$ (the converse implication is the classical Jacobian conjecture, in the case where $R$ is a field of characteristic $0$). We obtain a natural group homomorphism
\[\Jac\colon \Aut_{R}(\A^n) \to R^*\]
whose kernel is the group $\SAut_R(\A^n)$ of \emph{automorphisms of Jacobian $1$}. Apart from this subgroup, the group $\Aut_{R}(\A^n)$ contains some classical subgroups, of respectively linear automorphisms, translations, affine automorphisms and triangular automorphisms:
\[\begin{array}{rcl}
\GL_n(R)&=&\{(a_{11}x_1+\cdots+a_{1n}x_n,\ldots,a_{n1}x_1+\cdots+a_{nn}x_n) \mid  \det((a_{ij})_{i,j=1}^n)\in  R^*\}\\
\Tra_n(R)&=& \{(x_1+v_1,\ldots,x_n+v_n)\mid (v_1,\ldots,v_n)\in \k^n\}\\
\Aff_n(R)&=& \{\tau\circ \alpha\mid \alpha\in \GL_n(R), \tau\in \Tra_n(R)\}=\Tra_n(R)\rtimes \GL_n(R)\\
\Tri_n(R)&=& \{(a_1x_1+b_1, \ldots,a_nx_n+b_n)\mid a_i\in R^*,  b_i\in R[x_1,\ldots,x_{i-1}], 1\le i\le n\}
\end{array}\]
The group generated by affine and triangular automorphisms is the tame subgroup:
\[\Tame_R(\A^n)=\langle \Aff_n(R), \Tri_n(R)\rangle=\langle \GL_n(R), \Tri_n(R)\rangle.\]
The equality $\Tame_\k(\A^2)=\Aut_\k(\A^2)$ is known to be true for any field $\k$ (Jung-Van der Kulk's theorem, \cite{Jung}). However, $\Tame_{\k[x]}(\A^2)\subsetneq \Aut_{\k[x]}(\A^2)$ \cite[Theorem 1.4]{Nagata} and $\Tame_{\k}(\A^3)\subsetneq \Aut_{\k}(\A^3)$ if $\mathrm{char}(\k)=0$  \cite{ShestakovUmirbaev}.

If $B$ is an integral domain and $A$ is a subring of $B$, we have an inclustion $\Aut_A(\A^n)\subseteq \Aut_B(\A^n)$ for each $n\ge 1$. We will often study this for the rings $\k\subseteq \k[t]\subseteq \k[t,\frac{1}{t}]$, where $\k$ is a field. If $g\in \Aut_{\k[t]}(\A^n)$, one can replace $t$ with any element of $\k$ and obtain an element of $\Aut_\k(\A^n)$. Hence, $g$ corresponds to a map
\[\A^1(\k)\to \Aut_\k(\A^n),\]
that corresponds actually a morphism, when $\Aut_\k(\A^n)$ is endowed with the structure of an \emph{ind}-group (see \cite{FurterKraft}). For this, we fix a degree $d$, define
\[\End_{\le d}(\A^n)=\{f=(f_1,\ldots,f_n)\in \End(\A^n) \mid \deg(f_i)\le d \forall i\},\]
that has the structure of an affine space defined over any field and observe that we have closed embeddings
\[\End_{\le 1}(\A^n)\hookrightarrow \End_{\le 2}(\A^n)\hookrightarrow \End_{\le 3}(\A^n)\hookrightarrow \cdots\]
We may then put a topology on $\End_\k(\A^n)$, for each infinite field $\k$, such that a subset is closed if its intersection with any $\End_{\le d}(\A^n)$ is closed. This makes $\Aut_\k(\A^n)$ locally closed in $\End_\k(\A^n)$ and $\SAut_\k(\A^n)$ to be closed. Moreover, the topology is also the strongest one that makes all morphisms $A\to \Aut_\k(\A^n)$ induced by elements of $\Aut_A(\A^n)$ continuous (see \cite{Blanc}).  As explained before, there is a weaker topology one can also put on $\End_\k(\A^n)$ (see \cite{Stampfli}). Here, $\SAut_\k(\A^n)$ is still closed in both $\Aut_\k(\A^n)$ and $\End_\k(\A^n)$.

\section{Translations - Centralisers and normal subgroups}
As explained in the introduction, we say that an element of $\Aut_R(\A^n)$ (where $n\ge 1$ is an integer and $R$ is an integral domain) is a \emph{translation} if it is of the form
\[(x_1+c_1,x_2+c_2,\ldots,x_n+c_n)\]
where $c_1,\ldots,c_n\in R$.
\subsection{Centralisers of translations}

A simple calculation (Lemma~\ref{Lem:CentraliserTranslations} below) shows that the group of translations is its own centraliser in $\End_R(\A^n)$ when $R$ is an infinite integral domain. For this, we first need the following technical lemma.
\begin{lemma}\label{lem:AdditiveSet}
Let $R$ be an infinite integral domain, let $n\ge 1$ be an integer and let $f\in R[x_1,\ldots,x_n]$ be a polynomial in $n$ variables such that
\[f(x_1+c,x_2,\ldots,x_n)=f(x_1,x_2,\ldots,x_n)\]
for each $c\in R$. Then $f\in R[x_2,\ldots,x_n]$.
\end{lemma}
\begin{proof}
We add one variable and consider the polynomial
\[g=f(x_1+x_{n+1},x_2,\ldots,x_n)-f(x_1,x_2,\ldots,x_n)\in R[x_1,\ldots,x_{n+1}].\]
Replacing $x_{n+1}$ with any element $c\in R$, the polynomial $g$ becomes zero. We now prove that this implies that $g$ is the zero polynomial. For this, we embedd $R[x_1,\ldots,x_{n+1}]$ into $K[x_{n+1}]$, where $K$ is the field of fraction of the integral domain $R[x_1,\ldots,x_n]$. After doing this, $g$ becomes a polynomial in one variable over an integral domain, which has infinitely many roots and is thus the zero polynomial.

We have then proven the following equality in  $R[x_1,\ldots,x_{n+1}]$:
\[f(x_1+x_{n+1},x_2,\ldots,x_n)=f(x_1,x_2,\ldots,x_n).\]
Replacing $x_1$ with $0$ and $x_{n+1}$ with $x_1$, we obtain
\[f(x_{1},x_2,\ldots,x_n)=f(0,x_2,\ldots,x_n)\in R[x_2,\ldots,x_n].\qedhere\]
\end{proof}

\begin{lemma}\label{Lem:CentraliserTranslations}
Let $R$ be an infinite integral domain, let $n\ge 1$ be an integer and let $f\in \End_R(\A^n)$ be an element that commutes with all translations. Then, $f$ is itself a translation.

In other words, the group of translations is its own centraliser in $\End_R(\A^n)$.
\end{lemma}
\begin{proof}
We write $f=(f_1,\ldots,f_n)$, with $f_1,\ldots,f_n\in R[x_1,\ldots,x_n]$. 

For each $c\in R$, the endomorphism $f$ commutes with $\tau=(x_1+c,x_2,\ldots,x_n)$. The $n$ coordinates of $f\circ\tau=\tau \circ f$ give the following $n$ equations:
\begin{align*}f_1(x_1+c,x_2,\ldots,x_n)&=f_1(x_1,x_2,\ldots,x_n)+c\\
f_i(x_1+c,x_2,\ldots,x_n)&=f_i(x_1,x_2, \ldots,x_n)\text{ for each }i\in \{2,\ldots,n\}.\end{align*}
This implies that $f_1-x_1,f_2,\ldots,f_n$ all belong to the set\[\Delta=\{q\in R[x_1,\ldots,x_n] \mid q(x_1+c,x_2,\ldots,x_n)\text{ for each }c\in R\}.\]
Lemma~\ref{lem:AdditiveSet} implies that $\Delta=R[x_2,\ldots,x_n]$. So we have found that \[f_1-x_1,f_2,\ldots,f_n\in R[x_2,\ldots,x_n].\]
In other words, these polynomials do not depend on $x_1$.

For each $i\in \{2,\ldots,n\}$, we  use the fact that $f$ commutes with \[\tau=(x_1,\ldots,x_{i-1},x_i+c,x_{i+1},\ldots,x_n)\] for each $c\in R$ and obtain similarly that 
\[f_1,\ldots,f_{i-1},f_i-x_i,f_{i+1},\ldots,f_n\in R[x_1,\ldots,x_{i-1},x_{i+1},\ldots,x_n],\]
i.e.~these polynomials do not depend on $x_i$.

Hence, for each $j\in \{1,\ldots,n\}$, we have proven that $f_j$ does not depend on the other variables, so $f_j\in R[x_j]$, and moreover that $f_j-x_j$ does not depend on $x_j$, so $f_j-x_j\in R$. This implies that $f$ is a translation.
\end{proof}
\begin{example}
Let $\k$ be a finite field with $q$ elements. The automorphism
\[f=(x_1+x_2-x_2^q,x_2)\in \SAut_{\k}(\A^2)\]
commutes with all translations. The hypothesis that $R$ is infinite is then necessary in Lemma~\ref{Lem:CentraliserTranslations}. Note that $f$ acts trivially on the $q^2$ elements of the set $\A^2(\k)$ of $\k$-rational points of $\A^2$.
\end{example}
\subsection{Bijections  on finite sets of points}
In \cite{MaubachFinite}, the action of $\Aut_\k(\A^n)$ on the set $\A^n(\k)$ has been studied by Stefan Maubach, for any finite field $\k=\F_q$ with $q$ elements. The main result consists of showing that all permutations are obtained if $q=2$ or $q$ is odd, and that if $q\ge 4$ is even, the image of the group of tame automorphisms $\Tame_\k(\A^n)$ in the group of permutations of $\A^n(\k)$ is the alternating group. 

The result of Maubach gives an interesting way to search for non-tame automorphisms over finite fields. The existence of an element of $\Aut_\k(\A^n)$ acting by an odd permutation on $\A^n(\k)$, where $k=\F_{2^r}$ with $r\ge 2$ would yield such an element.

We now consider the action of $\SAut_\k(\A^n)$ and in particular of the translations. 
\begin{lemma}\label{EvenAction}
Let $\k=\F_q$ be a finite field with $q$ elements, let $n\ge 1$ be an integer and let us consider the action of the group
\[\SAut_\k(\A^n)\cap \Tame_\k(\A^n)\]
of tame automorphisms of Jacobian $1$ on the set $\A^n(\k)\simeq \k^n$. Then, the following hold:
\begin{enumerate}
\item\label{TranslationsAreEven}
The action of every translation is  even, i.e.~contained in the alternating group $\mathrm{Alt}(\k^n)$, if and only $(q,n)\not=(2,1)$.
\item\label{sTameEven}
The action of every element of $\SAut_\k(\A^n)\cap \Tame_\k(\A^n)$ is  even if and only if $q>2$.
\end{enumerate}
\end{lemma}
\begin{proof}
We write $q=p^r$ where $p=\mathrm{char}(\k)$ is  a prime number and $r\ge 1$.

\ref{TranslationsAreEven}: A non-trivial translation is equal to $(x_1+v_1,\ldots,x_n+v_n)$ for some $(v_1,\ldots,v_n)\in \k^n\setminus \{0\}$. The order of such an element is then equal to $p$. If $p$ is odd, the induced permutation is then even. We may thus assume that $p=2$. In that case, the orbits have all size $2$ and the number of orbits is then $\frac{\lvert \k\rvert^n}{2}=\frac{(2^r)^n}{2}=2^{rn-1}$. We then obtain an even permutation, except when $rn=1$. This latter case corresponds to $(q,n)=(2,1)$ and to the translation $(x_1+1)$ of $\A^1_{\F_2}$.

\ref{sTameEven}: If $q=2$, we consider $(x_1+x_2x_3\cdots x_n,x_2,\ldots,x_n)$ (which corresponds to $(x_1+1)$ if $n=1$), that is a tame automorphism of Jacobian $1$. It is of order $2$ and permutes exactly two points, namely $(1,1,\ldots,1)$ with $(0,1,\ldots,1)$ (this corresponds to $(0)$ and $(1)$ if $n=1$). It is thus an odd permutation.

We now assume that $q>2$ and prove that the action of every element of $\SAut_\k(\A^n)\cap \Tame_\k(\A^n)$ is always even. This group is generated by $\SL_n(\k)$ and by elements of the form \[e_s=(x_1+s(x_2,\ldots,x_n),x_2,\ldots,x_n),\] where $s\in \k[x_2,\ldots,x_n]$. Moreover, $\SL_n(\k)$ is generated by elementary matrices conjugate to $e_s$ with $s$ linear. It then suffices to show that the action of every $e_s$ is even. As before, we may assume that $p=2$, as the order of $e_s$ is $p$. Every orbit has then order $2$ and it remains to see that the number of orbits is even, or equivalently that the number of points moved is divisible by $4$. As $4$ divides $q$, it divides $\lvert \A^n(\k)\rvert$. We then only need to see that the number of fixed points is divisible by $4$. The set of fixed points being $\A^1\times X$ where $X$ is the zero locus of $s(x_2,\ldots,x_n)=0$, this is again given by the fact that $4$ divides $q=\lvert \A^1(\k)\rvert$.
\end{proof}
\begin{remark}
Lemma~\ref{EvenAction} could yield to a way to look for non-tame automorphisms. For each finite field $\k=\F_q$ with $q>2$, the existence of an element of $\SAut_{\k}(\A^n)$ whose action on $\A^2(\k)$ is odd would yield a non-tame automorphism. 
\end{remark}
\subsection{Normal subgroups containing affine automorphisms}\label{NormalTransLin}
In this section, we prove that the normal subgroups generated by all translations or by linear elements contains all tame automorphisms. The result has already be proven if $n=2$ and $\k=\mathbb{C}$ in \cite[Lemma~30]{FurterLamy}, or in \cite{Lewis} for infinite fields. We recall some of the arguments of \cite{Lewis} for self-containedness and give small complements.

The following proposition corresponds to \cite[Theorem 6]{Lewis} in the case where $\lvert \k\rvert\not=2$; the proof for $\mathrm{char}(\k)\not=2$ is similar, the proof for $\mathrm{char}(\k)=2$ is a bit shorter here. The case where $\lvert \k\rvert=2$ answers the question asked in \cite[Remark 1]{Lewis}.
\begin{proposition}\label{TranslationSLn}
Let $\k$ be a field and let $n\ge 1$ be an integer. 
The following conditions are equivalent:
\begin{enumerate}
\item\label{TranslationsSLn1}
There exists a normal subgroup $N\subseteq\SAut_\k(\A^n)$
that contains all translations but does not contain $\SL_n(\k)$.
\item\label{TranslationsSLn2}
$\lvert \k\rvert= 2$ and $n=2$.
\end{enumerate}
\end{proposition}
\begin{proof}
$\ref{TranslationsSLn2}\Rightarrow\ref{TranslationsSLn1}$: Suppose  that $n=2$ and $\k=\F_2$. The group $\SAut_\k(\A^2)$ then acts on the four points of $\A^2(\k)$, and the action of every translation is even (lemma~\ref{EvenAction}\ref{TranslationsAreEven}). However, the element $(x_1+x_2,x_2)$ permutes the two points $(0,1)$ and $(1,1)$ and fixes the points $(0,0)$ and $(1,0)$. It is then an odd permutation. The normal  subgroup $N\subseteq \SAut_\k(\A^2)$ consisting of automorphisms inducing an even permutation on the four points then contains all translations but does not contain $\SL_2(\F_2)$.

$\ref{TranslationsSLn1}\Rightarrow\ref{TranslationsSLn2}$: By contraposition, we assume that $n\not=2$ or that $\lvert \k\rvert\not=2$ and prove that every normal subgroup $N\subseteq \SAut_\k(\A^n)$ that contains all translations also contains $\SL_n(\k)$. If $n=1$, then $\SAut_\k(\A^n)$ only consists of translations, so the result holds. We may then assume $n\ge 2$.

For each $\epsilon=(\epsilon_2,\ldots,\epsilon_n)\in \k^{n-1}$, the translation \[\tau_\epsilon=(x_1,x_2+\epsilon_2,\ldots,x_n+\epsilon_n)\] belongs to $N$. Moreover, for each polynomial $q\in \k[x_2,\ldots,x_n]$, the element 
\[e_q=(x_1+q(x_2,\ldots,q_n),x_2,\ldots,x_n)\]
belongs to $\SAut_\k(\A^n)$. Hence, the commutator $h_{q,\epsilon}=(e_q^{-1}\circ \tau_\epsilon^{-1}\circ e_q)\circ \tau_\epsilon$  belongs to $N$. We write explicitely $h_{q,\epsilon}$ as
\[h_{q,\epsilon}=(x_1+q(x_2+\epsilon_2,\ldots,x_n+\epsilon_n)-q(x_2,\ldots,x_n),x_2,\ldots,x_n).\]

To prove that $\SL_n(\k)$ is contained in $N$, we only need to show that $e_{\lambda x_2}=(x_1+\lambda x_2,x_2,\ldots,x_n)$ belongs to $N$ for each $\lambda\in \k$. Indeed, every elementary matrix of $\SL_n(\k)$ with diagonal entries equal to $1$ and only one other non-zero entry is conjugated to some $e_{\lambda x_2}$ by permutations, and all of them generate $\SL_n(\k)$.

If $\mathrm{char}(\k)\not=2$, we choose $q=\frac{\lambda}{2}x_2^2$, $\epsilon_2=1$ and obtain $h_{q,\epsilon}=(x_1+\lambda x_2+\frac{\lambda}{2},x_2,\ldots,x_n)\in N$. Composing with the translation $(x_1-\frac{\lambda}{2},x_2,\ldots,x_n)$, we obtain $e_{\lambda x_2}\in N$ as desired.

If $n\ge 3$, we choose $q=x_2x_3$ and $(\epsilon_2,\epsilon_3)=(0,\lambda)$ and obtain $e_{\lambda x_2}=h_{q,\epsilon}\in N$ as desired.

The remaining case is when $n=2$ and  $\mathrm{char}(\k)=2$. By assumption, we have $\lvert \k\rvert\ge 3$. We choose $\theta,\mu,\nu\in \k$ and choose two possibilities for $(q,\epsilon_2)$, namely $(\theta \nu x_2^3,\mu)$ and $(\mu \nu x_2^3,\theta)$. This gives two elements for $h_{q,\epsilon}\in N$, namely
\begin{align*}
(x_1+\theta\mu\nu x_2^2+\theta\mu^2\nu x_2+\theta\mu^3\nu,x_2)\in N,\\
(x_1+\theta\mu\nu x_2^2+\theta^2\mu\nu x_2+\theta^3\mu\nu,x_2)\in N.
\end{align*}
The composition of these two elements with  the translation $(x_1+\theta\mu\nu(\theta^2+\mu^2),x_2,)$ gives
\[(x_1+\theta\mu\nu(\theta+\mu) x_2,x_2)\in N.\]
As $\lvert \k\rvert\ge 3$, we may choose $\theta,\mu\in \k^*$, both different and then $\nu=\frac{\lambda}{\theta\mu(\theta+\mu)}$ to obtain $e_{\lambda x_2}\in N$ as desired.
\end{proof}

The following result is due to Drew Lewis, we recall the proof, as it is short, for self-containedess. 

\begin{proposition}  \cite[Theorem 9]{Lewis}\label{SLnTame}
Let $\k$ be a infinite field and let $n\ge 2$ be an integer. Every normal subgroup
\[N\subseteq\SAut_\k(\A^n)\]
that contains $\SL_n(\k)$ also contains all tame automorphisms of Jacobian~$1$.
\end{proposition}

\begin{proof}
As the tame group is generated by $\SL_n(\k)$ and elementary automorphisms, it suffices to show that 
\[ e_s=(x_1+s(x_2,\ldots,x_n),x_2,\ldots,x_n)\in N\]
for each $s\in \k[x_2,\ldots,x_n]$. As $e_{s}\circ e_{s'}=e_{s+s'}$, we only need to obtain each monomial. As in the proof of  Proposition~\ref{TranslationSLn}, we use some commutators.

For each $\alpha=(\alpha_2,\ldots,\alpha_n)\in (\k^*)^{n-1}$, the diagonal matrix \[\delta=({(\alpha_2\cdots \alpha_n)^{-1}} x_1,\alpha_2 x_2,\ldots,\alpha_n x_n)\in \SL_n(\k)\] belongs to $N$. Moreover, for each polynomial $q\in \k[x_2,\ldots,x_n]$, the element 
\[e_q=(x_1+q(x_2,\ldots,q_n),x_2,\ldots,x_n)\]
belongs to $\SAut_\k(\A^n)$. Hence, the commutator $u_{q,\alpha}=(e_q^{-1}\circ \delta^{-1}\circ e_q)\circ \delta$  belongs to $N$. We write explicitely $u_{q,\alpha}$ as
\[u_{q,\alpha}=(x_1+\alpha_2\cdots \alpha_nq(\alpha_2 x_2,\ldots,\alpha_n x_n)-q(x_2,\ldots,x_n),x_2,\ldots,x_n).\]
For all integers $i_2,\ldots,i_n\ge 0$ we may choose $q(x_2,\ldots,x_n)=\xi x_2^{i_2}\cdots x_n^{i_n}$ with $\xi\in \k^*$ and obtain 
\[u_{q,\alpha}=(x_1+(\alpha_2^{i_2+1}\cdots \alpha_n^{i_n+1}-1)\xi x_2^{i_2}\cdots x_n^{i_n},x_2,\ldots,x_n)\in N.\]
As $\k$ is infinite, we may choose $\alpha_2,\ldots,\alpha_n\in \k^*$ such that  $\alpha_2^{i_2+1}\cdots \alpha_n^{i_n+1}\not =1$ 
and choose $\xi$ to obtain
\[(x_1+\lambda x_2^{i_2}\cdots x_n^{i_n},x_2,\ldots,x_n)\in N\]
for each $\lambda\in \k$. This happens for each monomial, so we obtain all tame automorphisms.
\end{proof}

Proposition~\ref{PropLewis} now follows from Propositions~\ref{TranslationSLn} and~\ref{SLnTame}.

We now prove that the group $N$ of Proposition~\ref{PropLewis} not only contains all tame automorphisms but also contains the Nagata automorphism. The calculation comes from \cite{MaubachPoloni} and may be actually generalised to many non-tame automorphisms (see \cite[Theorem 2.6]{MaubachPoloni}). It was also included in \cite[Observation 4]{Zygadlo} and in \cite[Proposition 15.11.5]{FurterKraft}.
\begin{example}\label{NagataAuto}All calculations given below are inspired from \cite[Theorem 3.3]{MaubachPoloni}, but one does not need to read the above article to follow the calculations.
 Let $\k$ be a field and let us write $\Delta=x_1x_3+x_2^2$. For each $\alpha\in \k$, we define
\[N_\alpha=(x_1-2\alpha x_2\Delta-\alpha^2 x_3\Delta^2,x_2+x_3\alpha \Delta,x_3)\in \End_\k(\A^3)\]
One then calculate $(N_\alpha)^*(\Delta)=\Delta$ and then  $N_\alpha\circ N_\beta=N_{\alpha+\beta}$ for all $\alpha,\beta\in \k$, which implies that $N_\alpha\in \Aut(\A^3)$ for each $\alpha\in \k$, with $(N_\alpha)^{-1}=N_{-\alpha}$. Moreover, $\Jac(N_\alpha)=1$ for each $\alpha\in \k^*$. The element $N_1\in \SAut_\k(\A^3)$ is the classical \emph{Nagata automorphism} of \cite[Section 2.1]{Nagata}.

For each $u\in \k^*$, we write $L_{u}=(ux_1,x_2,u^{-1}x_3)\in \SL_3(\k)$ We again have $(L_u)^*(\Delta)=\Delta$ and then obtain that
\[(L_u)^{-1}\circ N_\alpha \circ L_u=N_\frac{\alpha}{u}\]
for each $\alpha\in \k^*$. This implies that $N_\alpha\circ L_{u}\circ N_{-\frac{\alpha}{u}}=L_{u}$, and thus \[N_\alpha\circ (L_{u}\circ N_{\frac{\alpha(u-1)}{u}})\circ N_\alpha^{-1}=L_{u},\]
which implies that $L_{u}\circ N_{\frac{\alpha(u-1)}{u}}$ is conjugate to $L_{u}$ in $\SAut_{\k}(\A^3)$, so $N_{\frac{\alpha(u-1)}{u}}$ is in the normal subgroup generated by $\SL_3(\k)$. As soon as $\lvert \k\rvert>2$, we may choose $u\in \k\setminus\{0,1\}$ and choose $\alpha\in \k^*$ so that $\frac{\alpha(u-1)}{u}$ is any possible value in $\k^*$. Hence $N_\beta$ is in the normal subgroup generated by $\SL_n(\k)$ for each $\beta\in \k^*$. In particular, the Nagata automorphism is contained in any normal subgroup of $\SAut_\k(\A^3)$ that contains $\SL_3(\k)$.
\end{example}
\subsection{Quotients of $\SAut_\k(\A^2)$ when $\k$ is finite}\label{SecFiniteA2}
The proof of Proposition~\ref{SLnTame} (second part of Proposition~\ref{PropLewis}) really needs that the field is infinite. In this section, we prove that the results are actually false over any finite field. Indeed, we provide a surjective group homomorphism $\SAut_\k(\A^2)\to \k[x]$ whose kernel contains all translations, and also all linear maps if $\lvert \k\rvert>2$, for each finite field $\k$. This will give the proof of Proposition~\ref{PropFiniteA2} and show a very different behaviour from the case of infinite fields, where $\SAut_\k(\A^2)$ is a perfect group, and where the normal subgroup generated by linear maps or translations is everything (Proposition~\ref{PropLewis}). The construction uses the following natural vector space $V$.
\begin{lemma}\label{lemmaV}
Let $\k$ be a field and let $V\subseteq \k[x]$ be the $\k$-linear subspace generated by 
\[\{s-as(ax+b)\mid s\in\k[x], a\in\k^*, b\in\k\}.\]
Then, the following hold:
\begin{enumerate}
\item\label{kV}
$\k\subseteq V$;
\item\label{Vallfinite}
If $\k$ is infinite, then $V=\k[x]$.
\item\label{Vnotallfinite}
If $\k$ is finite with $q$ elements, the classes of $(x^{q-1}(x^q-x)^{m(q-1)-1})_{m\ge 1}$ in $\k[x]/V$ are linearly independent. In particular, $\k[x]/V$ has infinite dimension.
\item\label{kmore2Linear}
If $\lvert \k\rvert>2$, then $cx+d\in V$ for all $c,d\in \k$.
\item\label{k2V}
If $\k=\F_2$, then $V=\k[x(x+1)]$, so $x\not\in V$.
\end{enumerate}
\end{lemma}
\begin{proof}
\ref{kV}-\ref{Vallfinite}-\ref{kmore2Linear}: For each $i\ge 0$, if $a\in \k^*$ is such that  $a^{i+1}\not=0$, we may choose $s=x^i$, $b=0$ and obtain 
\[s-as(ax+b)=(1-a^{i+1}) x^i\in V.\]
Choosing $i=0$, we obtain $1\in V$, which proves \ref{kV}. If $\k$ is infinite, $V$ contains all monomials and thus $V=\k[x]$, whence \ref{Vallfinite}. If $\lvert \k\rvert\not=2$, with $i=1$ we obtain $x\in V$, which proves \ref{kmore2Linear}.

\ref{Vnotallfinite}:
Suppose now that $\k$ is a finite field with $q$ elements. We consider the polynomial \[R=x^q-x=\prod_{\nu\in\k} (x-\nu).\] For all $a\in\k^*$ and $b\in\k$, the map $x\mapsto ax+b$ is an automorphism of $\A^1_\k$, so $R$ and $R(ax+b)$ have the same roots. This gives $R(ax+b)=a^qR=aR$.

The elements $m_{i,j}=x^iR^j$, with $0\le i\le q-1$ and $j\ge 0$ form a basis of $\k[x]$ as they have all possible degrees. Any element of $V$ is then a linear combination of elements of the form
\[m_{i,j}-am_{i,j}(ax+b)=(x^i-a^{j+1}(ax+b)^i)R^j.\]
Such a polynomial is a linear combination of elements of the form $m_{u,j}$ with $0\le u\le i$, and the coefficient of $m_{i,j}$ is $1-a^{i+j+1}$. Hence, the coefficient of any polynomial $m_{q-1,j}$ of any element of $V$ is zero as soon as $a^{q+j}=1$ for each element $a\in\k^*$, i.e.~when $q+j$ is a multiple of $q-1=\lvert\k\rvert$. This corresponds to ask that $j=m(q-1)-1$ for some $m\ge 1$, and thus proves \ref{Vnotallfinite}.

\ref{k2V}: If $\k=\F_2$, then $V$ is the linear subspace generated by $s+s(x+1)$ with $s\in \k[x]$. The polynomial $R$ is now $R=x^2+x=x(x+1)$. As the polynomials $x^iR^j$ with $i\in \{0,1\}$ and $j\ge 0$ form a basis of $\k[x]$, we obtain
\[\k[x]=\k[R]\oplus x\k[R].\]
Hence, every  $s\in\k[x]$ can be written as $s=c+dx$ with $c,d\in \k[R]$. This gives $s(x)+s(x+1)=d$ and thus proves that $V=\k[R]$.
\end{proof}
\begin{proposition}\label{Prop:SautFinite}
Let $\k$ be a finite field and let $V\subseteq \k[x]$ be the $\k$-linear subspace generated by 
\[\{s-as(ax+b)\mid s\in\k[x], a\in\k^*, b\in\k\}.\]
Then, we have a unique surjective group homomorphism
\[\rho\colon \SAut_\k(\A^2)\twoheadrightarrow \k[x]/V\]
that sends $(x_1,x_2+s(x_1))$ onto the class of $s$ for each $s\in\k[x]$. 
\end{proposition}
\begin{proof}
We first prove that if $\rho$ exists, it is  unique and surjective. For each $s\in \k[x]$, the elements
\[e_s=(x_1,x_2+s(x_1))\text{ and }\quad e_s'=(x_1-s(x_2),x_2)\] are conjugated by $(-x_2,x_1)\in \SL_2(\k)$. So $\rho(e_s)=\rho(e_s')$, is the class of $s$ in $\k[x]/V$. As $\SAut_\k(\A^2)$ is generated by elements of the form $e_s$ and $e_s'$, this shows that $\rho$ is unique. It also shows that $\rho$ is surjective, as we may choose any $s\in \k[x]$.

To show that the definition on all $e_s$ extends to a well-defined group homomorphism $\rho$, we define two group homomorphisms  on $T$ and $A$, the triangular and affine groups:
\[\rho_T\colon T\to \k[x]/V, \quad \rho_A\colon A\to \k[x]/V\] and check that $\rho_T|_{A\cap T}=\rho_A|_{A\cap T}$. As $\SAut_\k(\A^2)$ is the amalgamated product of $T$ and $A$ over their intersection, this will give the result.

We first consider the triangular group $T=\Tri_2(\k)\cap \SAut_\k(\A^2)$:
\[T=\{(ax_1+b,a^{-1}x_2+s(x_1))\mid a\in \k^*, b\in \k, s\in \k[x]\}\simeq \k[x]\rtimes \Aut(\A^1).\]
The action of $\Aut(\A^1)$ on $\k[x]$ is the following one: $ax_1+b\in \Aut(\A^1)$ acts on $s\in \k[x]$ as $s\mapsto a s(ax+b)$. We define $\rho_T((ax_1+b,a^{-1}x_2+s(x_1)))$ to be the class of $s$. To check that this is a well-defined group homomorphism, we just need to check that $s$ and $s(ax+b)$ have the same class, which is exactly given by the definition of $V$. Moreover, $\rho_T(e_s)$ is indeed the class of $s$, for each $s\in \k[x]$, as desired.

Now we go to the affine group  $A=\Aff_2(\k)\cap \SAut_\k(\A^2)$:
\[A=\{(ax_1+bx_2+\tau,cx_2+dx_1+\epsilon)\mid \left(\begin{smallmatrix} a & b\\ c & d\end{smallmatrix}\right)\in \SL_2(\k), (\tau,\epsilon)\in \k^2\}\]
We then have $A=\SL_2(\k)\rtimes (\k^2,+)$.
If $\lvert \k\rvert>2$, we define $\rho_A(A)=\{0\}$.  If $\lvert \k\rvert=2$, then $A$ acts on the four points of $\A^2(\F_2)$. We send an element of $A$ onto $x$ or $0$, depending if the permutation is odd or even. This gives a group homomorphism $\rho_A\colon A\to \k[x]/V$.

We have now defined $\rho_T$ and $\rho_A$ and prove that $\rho_T|_{A\cap T}=\rho_A|_{A\cap T}$. The group $A\cap T$ is equal to 
\[A\cap T=\{(ax_1+b,a^{-1}x_2+s(x_1))\mid a\in \k^*, b\in \k, s\in \k[x] \deg(s)\le 1\}.\]
The image $\rho_T(f)$ of $f=(ax_1+b,a^{-1}x_2+s(x_1))$ is the class of $s$. If $\lvert \k \rvert >2$, the class of $s$ is trivial since $\deg(s)\le 1$ (Lemma~\ref{lemmaV}\ref{kmore2Linear}), so we obtain $\rho_T(f)=\rho_A(f)=0$.
 If $\lvert \k \rvert=2$, then $f=(x_1+b,x_2+s(x_1))$ and the translations being even (Lemma~\ref{EvenAction}\ref{TranslationsAreEven}), the element $\rho_A(f)$ is equal to $\rho_A((x_1,x_2+x_1))$ if $\deg(f)=1$ and to $0$ if $\deg(f)\le 0$. This implies that $\rho_A(f)=\rho_T(f)$ as $\k\subseteq V$ (Lemma~\ref{lemmaV}\ref{kV}).
\end{proof}
We can now give the proof of Proposition~\ref{PropFiniteA2}:
\begin{proof}[Proof of Proposition~\ref{PropFiniteA2}]
As $\k$ is finite,  Proposition~\ref{Prop:SautFinite} gives a unique surjective group homomorphism $\rho\colon \SAut_\k(\A^2)\twoheadrightarrow \k[x]/V$ that sends $(x_1,x_2+s(x_1))$ onto the class of $s$ for each $s\in\k[x]$.

By Lemma~\ref{lemmaV}\ref{kV}, every translation is in the kernel. Moreover, $\SL_2(\k)$ is in the kernel if and only if $\lvert \k\rvert>2$ (Lemma~\ref{lemmaV}\ref{kmore2Linear}-\ref{k2V}).

As $\k[x]/V$ is infinite dimensional (Lemma~\ref{lemmaV}), it is isomorphic to $\k[x]$ as a $\k$-vector space. This gives a surjective group homomorphism $\tau\colon \SAut_\k(\A^2)\twoheadrightarrow (\k[x],+)$ whose kernel contains all translations. The kernel moreover contains $\SL_2(\k)$ if and only if $\lvert \k\rvert>2$. In the case where $\lvert \k\rvert=2$, we may also get a similar group homomorphism $\tau'\colon \SAut_\k(\A^2)\twoheadrightarrow (\k[x],+)$ whose kernel contains also $\SL_2(\k)$  by taking $\k[x]/V\to \k[x](V+\k x)$.
\end{proof}

\section{Families of automorphisms}
Before working with elements of $\Aut_{\k[t,\frac{1}{t}]}(\A^n)$ and $\Aut_{\k[t]}(\A^n)$, where $\k$ is a field, we recall the following lemma, probably known to all specialists. We give back a proof for a lack of reference. We will apply the next lemma to the case where $A=\k[t]$ and $B=\k[t,\frac{1}{t}]$.
\begin{lemma}\label{lem:ABJac}
Let $B$ be an integral domain and let $A\subseteq B$ be a subring. For each integer $n\ge 1$, we have
\[\Aut_A(\A^n)=\{ f\in \End_A(\A^n)\cap \Aut_B(\A^n) \mid \Jac(f)\in A^*\}.\]
\end{lemma}
\begin{proof}
As $\Aut_A(\A^n)\subseteq \End_A(\A^n)\cap \Aut_B(\A^n)$ and as the Jacobian of an element of $\Aut_A(\A^n)$ belongs to $A^*$, we have the inclusion
\[\Aut_A(\A^n)\subseteq\{ f\in \End_A(\A^n)\cap \Aut_B(\A^n) \mid \Jac(f)\in A^*\}.\]
Conversely, let us take an element $f\in \End_A(\A^n)\cap \Aut_B(\A^n)$ with $\Jac(f)\in A^*$. We want to prove that $f\in\Aut_A(\A^n)$. Replacing $f$ with $\tau\circ f$ where $\tau\in \Aut_A(\A^n)$ is a translation, we may assume that $f$ fixes the origin $(0,\ldots,0)\in \A^n$. We then write $f=(f_1,\ldots,f_n)$ with $f_i\in A[x_1,\ldots,x_n]$ and may write $f_i=f_{i,1}+f_{i,2}+\cdots+f_{i,d}$ where each $f_{i,j}\in A[x_1,\ldots,x_n]$ is homogeneous of degree $j$, and $d$ is the maximal degree of all polynomials $f_1,\ldots,f_n$.

Replacing $x_1,\ldots,x_n$ with $0$ in the derivative $Df=(\frac{\partial f_i}{\partial x_j})_{i,j=1}^n$ of $f$ gives $M=(Df)_{(0,\ldots,0)}=(\frac{\partial f_{i,1}}{\partial x_j})_{i,j=1}^n$. As $\Jac(f)=\det(Df)\in A^*$, we find that $\det(M)=\Jac(f)\in A^*$, so $M\in \GL_n(A)$. Replacing $f$ with $(f_{1,1},\ldots,f_{n,1})^{-1}\circ f$, we may thus assume that $\Jac(f)=1$ and that $(f_{1,1},\ldots,f_{n,1})=(x_1,\ldots,x_n)$ (in this case, one often says that $f$ is \emph{tangent to the identity}).

As $f\in \Aut_B(\A^n)$, it admits an inverse $g=f^{-1}\in \Aut_B(\A^n)$, that we write similarly as $f$ as $g=(g_1,\ldots,g_n)$, where $g_i=g_{i,1}+g_{i,2}+\cdots+g_{i,D}$ and $g_{i,j}\in B[x_1,\ldots,x_n]$ is homogeneous of degree $j$, and $D$ is the maximal degree of all polynomials $g_1,\ldots,g_n$. We now prove, by induction on $j\ge 1$, that $g_{i,j}\in A[x_1,\ldots,x_n]$ for each $i\in \{1,\ldots,n\}$. For this, we use that $g\circ f=(x_1,\ldots,x_n)$, which means that
\[g_i(f_1,\ldots,f_n)=x_i\]
for each $i\in \{1,\ldots,n\}$. This gives

\begin{equation}\label{xigf}\tag{$\spadesuit$}x_i=\sum_{m=1}^D g_{i,m}(f_1,\ldots,f_n)=\sum_{m=1}^D g_{i,m}(x_1+f_{1,2}+\cdots+f_{1,d},\ldots,x_n+f_{n,2}+\cdots+f_{n,d}).\end{equation}
The homogeneous part of degree $1$ of \eqref{xigf} gives $x_i=g_{i,1}(x_1,\ldots,x_n)$, which implies that $g_{i,1}=x_i\in A[x_1,\ldots,x_n]$ for each $i\in \{1,\ldots,n\}$. We then assume that $j>1$ and prove that $g_{i,j}\in A[x_1,\ldots,x_n]$, assuming that $g_{i,m}\in A[x_1,\ldots,x_n]$ for each $i\in \{1,\ldots,n\}$ and each $m<j$. The homogeneous part of degree $j$ of \eqref{xigf}  gives $0=\delta+g_{i,j}(x_1,\ldots,x_n)$, where $\delta$ is the homogeneous part of degree $i$ of $\sum_{m=1}^{j-1} g_{i,m}(f_1,\ldots,f_n)\in A[x_1,\ldots,x_n]$. This proves that $g_{i,j}=-\delta\in A[x_1,\ldots,x_n]$ as desired.
\end{proof}
\subsection{Closed normal subgroups of $\SAut_\k(\A^n)$}
\begin{proposition}\label{TheProp}
Let $\k$ be a field, let $n\ge 1$ be an integer and let $g\in \Aut_{\k[t]}(\A^n)\setminus \{\mathrm{id}\}$ be  such that replacing $t$ with zero one obtains $g(0)=\mathrm{id}\in \Aut_{\k}(\A^n)$.

There exist coprime integers $a,b\ge 1$ such that the following hold:
\begin{enumerate}
\item\label{Prop1}
For each extension $\k\subseteq L$ and each $\epsilon=(\epsilon_1,\ldots,\epsilon_n)\in L^n$, writing 
\[\rho_\epsilon=(x_1+t^{-b}\epsilon_1,\ldots,x_n+t^{-b} \epsilon_n)\in \Aut_{L[t,\frac{1}{t}]}(\A^n),\]
the element 
\[h_\epsilon = (\rho_\epsilon)^{-1} \circ g(t^a)\circ (\rho_\epsilon)\]
belongs to $\Aut_{L[t]}(\A^n)$ and replacing $t$ with zero in $h_\epsilon$ gives a translation
\[h_\epsilon(0)\in \Aut_L(\A^n).\]
Here, $g(t^a)\in \Aut_{\k[t^a]}(\A^n)$ is the element obtained by replacing $t$ with $t^a$ in $g$.
\item\label{Prop2}
The set $\{\epsilon\in \A^n\mid h_\epsilon(0)\not=\mathrm{id}\}$ is a dense open subset of $\A^n$. In particular, for each infinite extension $\k\subseteq L$, there exists $(\epsilon_1,\ldots,\epsilon_n)\in L^n$ such that $h_\epsilon(0)\in \Aut_L(\A^n)$ is a non-trivial translation.
\end{enumerate}
\end{proposition}
\begin{proof}
We write $g=(g_1,\ldots,g_n)$. For each $i\in \{1,\ldots,n\}$, the element $g_i\in \k[t,x_1,\ldots,x_n]$ is such that $g_i(0)=x_i$, where we replace here $t$ with $0$. We can then uniquely write $g_i$ as
\[g_i=x_i+\sum_{j\ge 1, m\ge 0} q_{i,j,m}t^j\]
where each $q_{i,j,m}\in \k[x_1,\ldots,x_n]$ is homogeneous of degree $m$. 

We now choose the two coprime integers $a,b\ge 1$ such that
\[\frac{a}{b}=\max \left.\left\{\frac{m}{j} \right| q_{i,j,m}\not=0\right\}.\]
This is possible since $g$ is not the identity, so at least one $q_{i,j,m}$ is not zero.

For each field extension $\k\subseteq L$ and each $\epsilon=(\epsilon_1,\ldots,\epsilon_n)\in L^n$, we write  
\[\rho_\epsilon=(x_1+t^{-b}\epsilon_1,\ldots,x_n+t^{-b} \epsilon_n)\in \Aut_{L[t,\frac{1}{t}]}(\A^n),\]
as in the statement of the proposition. Replacing $t$ with $t^a$ in $g$, we obtain $g(t^a)\in \Aut_{\k[t^a]}(\A^n)\subseteq \Aut_{\k[t]}(\A^n)$. Conjugating this element by $\rho_\epsilon$, we obtain 
\[h_\epsilon = (\rho_\epsilon)^{-1} \circ g(t^a)\circ (\rho_\epsilon)\in \Aut_{L[t,\frac{1}{t}]}(\A^n)\]
We now prove that $h_\epsilon$ is actually an element of $\Aut_{L[t]}(\A^n)$ and that replacing $t$ with $0$, we obtain a translation $h_\epsilon(0)\in \Aut_L(\A^n)$.

For $i\in \{1,\ldots,n\}$, the $i$-th coefficient of $h$ is
\[h_i=x_i+\sum_{j\ge 1, m\ge 0} q_{i,j,m}(x_1+t^{-b}\epsilon_1,\ldots,x_n+t^{-b} \epsilon_n)\cdot t^{aj}\]
As $aj-bm\ge 0$ if $q_{i,j,m}\not=0$ (this inequality is equivalent to $ \frac{m}{j}\le \frac{a}{b}$), we find that $h_i\in L[t,x_1,\ldots,x_n]$, i.e.~there is no denominator in $t$ in $h_i$. Moreover, replacing $t$ with $0$ in $h_i$ gives
\[h_i(0)=x_i+\sum_{\frac{m}{j}=\frac{a}{b} } q_{i,j,m}(\epsilon_1,\ldots, \epsilon_n)=x_i+c_i,\]
where $c_i\in L$.

The Jacobian of $g$ being an invertible element of $\k[t]$, it belongs to $\k^*$. Hence, the Jacobian of $h_\epsilon$ is the same element of $\k^*$, which implies that $h_\epsilon\in \Aut_{L[t]}(\A^n)$ (Lemma~\ref{lem:ABJac} applied to the integral domains $A=L[t]\subseteq B=L[t,\frac{1}{t}]$). 

Replacing $t$ with $0$ in $h$ give a translation $h_\epsilon(0)\in \Aut_L(\A^n)$, as $h_i(0)=x_i+c_i$ for each $i\in \{1,\ldots,n\}$. This proves then~\ref{Prop1}.

To prove~\ref{Prop2}, we observe that the translation $h_\epsilon(0)$ is the identity if and only if $c_i=0$ for each $i$. This corresponds to ask that 
\[\sum_{\frac{m}{j}=\frac{a}{b} } q_{i,j,m}(\epsilon_1,\ldots, \epsilon_n)=0\]
for each $i\in \{1,\ldots,n\}$ and thus gives an open subset of $\A^n$. It remains to see that this open subset is not empty, which corresponds to show that some $i\in \{1,\ldots,n\}$ exists such that the polynomial $P_i=\sum\limits_{\frac{m}{j}=\frac{a}{b} } q_{i,j,m}$ is not the zero polynomial. By definition of $a$ and $b$, there exists $i,j,m$ such that $q_{i,j,m}\not=0$ and such that $\frac{m}{j}=\frac{a}{b}$. We fix such an $i$ and then observe that all polynomials $q_{i,j,m}$ with $\frac{m}{j}=\frac{a}{b}$ have distinct degrees, as $m$ and $\frac{a}{b}$ determine $j=\frac{bm}{a}$. Hence all non-zero polynomials appearing in the sum defining $P_i$ are linearly independent over $\k$ and thus $P_i$  is not the zero polynomial. This gives the fact that open subset of~\ref{Prop2} is not empty, and hence contains points over each infinite extension of $\k$, which proves~\ref{Prop2}.
\end{proof}

\begin{proposition}\label{Prop:normalclosedExpli}
Let $\k$ be an infinite field, let $n\ge 2$ and let $f\in \SAut_\k(\A^n)\setminus \Tra_n(\k)$. There exists $h\in \SAut_{\k[t]}(\A^n)$ satisfying the following:

\begin{enumerate}
\item
For each $t_0\in \k^*$, replacing $t$ with $t_0$ gives an element  
\[h(t_0)\in\{\rho (\tau^{-1} f^{-1} \tau f) \rho^{-1}\mid \rho,\tau \in \Tra_n(\k)\}.\]
\item
Replacing $t$ with $0$ in $h$ gives $h(0)\in  \Tra_n(\k)\setminus \{\mathrm{id}\}$.
\end{enumerate}
\end{proposition}

\begin{proof}
As $\k$ is infinite, there exists  a translation that does not commute with $f$ (Lemma~\ref{Lem:CentraliserTranslations}). We write this translation as $(x_1+\nu_1,\ldots,x_n+\nu_n)$, with $\nu_1,\ldots,\nu_n\in \k$.

We then define $\tau=(x_1+t\nu_1,x_2+t\nu_2,\ldots,x_n+t\nu_n)\in \SAut_{\k[t]}(\A^n)$ and then consider
\[ g=\tau^{-1}\circ f^{-1}\circ\tau\circ f\in\SAut_{\k[t]}(\A^n).\]
Replacing $t$ with $1$ in $g$ gives a non-trivial element  $g(1)\in\SAut_\k(\A^n)\setminus \{\mathrm{id}\}$, since the translation $\tau(1)$ does not commute with $f$. Hence, $g$ is not the identity. Moreover, replacing $t$ with $0$ in $g$ gives $g(0)=\mathrm{id}$. We may thus apply Proposition~\ref{TheProp} and obtain integers $a,b\ge 1$ and $\epsilon=(\epsilon_1,\ldots,\epsilon_n)\in \A^n(\k)$ such that the following hold:
\[\rho=(x_1+t^{-b}\epsilon_1,\ldots,x_n+t^{-b} \epsilon_n)\in \SAut_{\k[t,\frac{1}{t}]}(\A^n),\]
is such that 
\[h = \rho^{-1} \circ g(t^a)\circ \rho\]
belongs to $\SAut_{\k[t]}(\A^n)$ and replacing $t$ with zero in $h$ gives a non-trivial translation
\[h(0)\in \SAut_\k(\A^n).\]
Replacing $t$ with $t_0\not=0$ in $h$, we obtain 
\[\rho(t_0)^{-1} \circ \tau(t_0^a)^{-1}\circ f \circ\tau(t_0^a)\circ \rho(t_0)\in \{\rho (\tau^{-1} f^{-1} \tau f) \rho^{-1}\mid \rho,\tau \in \Tra_n(\k)\}.\]\end{proof}

\begin{proof}[Proof of Proposition~\ref{PropLimitTrans}] Let $\k$ be an infinite field, let $n\ge 2$ be an integer and let $N\subseteq \SAut_\k(\A^n)$ be a closed normal subgroup with $N\not=\{\mathrm{id}\}$. 

We first prove that $N$ contains a non-trivial translation.
We choose an element $f\in N\setminus \{\mathrm{id}\}$. If $f$ is not a translation, we may apply Proposition~\ref{Prop:normalclosedExpli} and obtain $h\in \SAut_{\k[t]}(\A^n)$ as in the statement. As 
\[\A^1\to \SAut_{\k}(\A^n), t_0\mapsto h(t_0)\]
is a continuous map, the preimage of $N$ is closed. It contains the open subset $\A^1\setminus \{0\}$ and thus contains the whole $\A^1$. This proves that $N$ contains the non-trivial translation $h(0)$. 

As $n\ge 2$, all translations are conjugate in $\SAut_{\k}(\A^n)$. For this, one observes that $\SL_n(\k)$ acts transitively on translations by conjugation. Hence, all translations of $\SAut_{\k}(\A^n)$ are contained in $N$. 
\end{proof}

\subsection{Closed normal subgroups of $\Aut_\k(\A^n)$}
The case of $\Aut_\k(\A^n)$ is much more rigid than the case of $\SAut_\k(\A^n)$. There are much less closed normal subgroups.

Firstly, Proposition~\ref{SLnTame} may be generalised to $\Aut_{\k}(\A^n)$ instead of $\SAut_\k(\A^n)$. The proof is actually much simpler and with any field $\k$ with $\lvert \k\rvert>2$, contrary to the case of $\SAut_\k(\A^n)$ (see Proposition~\ref{PropFiniteA2}, proven in Section~\ref{SecFiniteA2}).  The result follows from \cite[Lemma 4.2]{MaubachPoloni} (see also \cite[Observation 2]{Zygadlo} when $\mathrm{char}(\k)=0$).
\begin{lemma}\cite[Lemma 4.2]{MaubachPoloni} Let $\k$ be a field with $\lvert \k\rvert>2$ and let $n\ge 1$ be an integer. Any normal subgroup $N\subseteq \Aut_\k(\A^n)$ containing $\GL_n(\k)$ also contains all tame automorphisms.
\end{lemma}
\begin{proof}
It suffices to show that show that 
\[ e_s=(x_1+s(x_2,\ldots,x_n),x_2,\ldots,x_n)\in N\]
for each $s\in \k[x_2,\ldots,x_n]$. Choose $a\in \k\setminus \{0,1\}$ (which is possible since $\lvert \k\rvert>2$), define $L=(ax_1,x_2,\ldots,x_n)$ and $b=\frac{1}{1-a}$. Then, it suffices to check that $e_s=L^{-1} \circ (e_{bs}^{-1}\circ L\circ  e_{bs})\in N$.
\end{proof}
In fact, adding the fact that $N$ is closed but removing that it is normal, Eric Edo obtained the following result:
\begin{proposition} \cite[Theorem 1.2]{Edo}
Let $\k$ be a field of characteristic $0$ and let $n\ge 2$.  Let $H\subseteq \Aut_\k(\A^n)$ be a closed subgroup with $\Aff_k(\A^n)\subsetneq H$. Then $H$ contains all tame automorphisms.
\end{proposition}

We now consider the fact where $N$ is closed and normal, but without the assumption that it contains an affine automorphism. This 
was done by Furter and Kraft in \cite[Proposition 15.11.5]{FurterKraft}, using the following classical ``Alexander trick'':
\begin{lemma}\label{AlexanderTrick}
Let $\k$ be a field, let $n\ge 1$ and let $f\in \Aut_\k(\A^n)$ that fixes the origin $(0,\ldots,0)\in \A^n$. Writing $\alpha=(tx_1,\ldots,tx_n)\in \Aut_{\k[t,\frac{1}{t}]}$, the element
\[g=\alpha^{-1} \circ f\circ \alpha\] 
is an element of $\Aut_{\k[t]}(\A^n)$ that is such that replacing $t$ with $0$ in $g$ gives $g(0)=(Df)_{(0,\ldots,0)}\in \GL_n(\k)$, which is the linear part of $f$, obtained by taking the homogenous part of degree $1$ of each coordinate of $f$.
\end{lemma}
\begin{proof}We write $f=(f_1,\ldots,f_n)$ with $f_i\in A[x_1,\ldots,x_n]$ and $f_i=f_{i,1}+f_{i,2}+\cdots+f_{i,d}$ where each $f_{i,j}\in A[x_1,\ldots,x_n]$ is homogeneous of degree $j$, and $d$ is the maximal degree of all polynomials $f_1,\ldots,f_n$. We then obtain 
\[g=(g_1,\ldots,g_n)\] 
where $g_i=g_{i,1}+tf_{i,2}+t^2f_{i,3}+\cdots+t^{d-1}f_{i,d}$ for each $i$.  Hence, $g \in \End_{\k[t]}(\A^n)\cap \Aut_{\k[t,\frac{1}{t}]}(\A^n)$. As $\Jac(g)=\Jac(f)\in\k^*$, we obtain $g\in \Aut_{\k[t]}$ by Lemma~\ref{lem:ABJac}.

Replacing $t$ with $0$ in $g$ gives $(f_{1,1},\ldots,f_{n,1})=(Df)_{(0,\ldots,0)}\in \GL_n(\k)$, which is the linear part of $f$.
\end{proof}
The following result was already given in \cite[Proposition 15.11.5]{FurterKraft} in the case where $\mathrm{char}(\k)=0$.
\begin{lemma}\label{NclosedSLn}
Let $\k$ be an infinite field, let $n\ge 2$ and let $N$ be a closed normal subgroup of $\Aut_\k(\A^n)$. If $N\not=\{\mathrm{id}\}$, then $N$ contains $\SL_n(\k)$.
\end{lemma}
\begin{proof}
We take $h\in N\setminus \{\mathrm{id}\}$ and choose two different points $p,q\in \A^n(\k)$ such that $h(p)=q$ (always possible as $\k$ is infinite). Replacing $h$ with a conjugate by a translation, we may assume that $p=(0,\ldots,0)$. Replacing with a conjugate by an element of $\GL_n(\k)$, we may assume that $q=(1,0,\ldots,0)$.

We choose  $\beta=(x_1,x_2+x_1(x_1-1)^2,x_3,\ldots,x_n)\in \SAut_\k(\A^n)$ that fixes both $p$ and $q$, and consider 
\[ f= h^{-1} \circ\beta^{-1}\circ  h\circ \beta \in N.\]
As $\beta$ fixes $p$ and $q$ and $h$ sends $p$ onto $q$, we find that $f(p)=p$, i.e.~$f$ fixes the origin. Moreover, the linear part of $f$ is the Jacobian matrix $Df=(\frac{\partial f_i}{\partial x_j})_{i,j=1}^n\in \GL_n(\k[x_1,\ldots,x_n])$ of $f$ evaluated at the origin $p$, that is equal to
\[ (Df)_p=(Dh^{-1})_q\cdot (D\beta^{-1})_q \cdot (Dh)_p \cdot (D\beta)_p\in \GL_n(\k).\]
As $h$ sends $p$ onto $q$ and $h^{-1}\circ h=\mathrm{id}$, we find that $(Dh^{-1})_q=( (Dh)_p)^{-1}$. Moreover, $(D\beta^{-1})_q=((D\beta)_q)^{-1}$ is the identity and $(D\beta)_p$ is a non-trivial elementary matrix, since $\frac{\partial x_1(x_1-1)^2}{\partial x_1}$ vanishes at $q$ but does not vanish at $p$. Hence, $(Df)_p=(D\beta)_p$ is a non-trivial elementary matrix in $\SL_n(\k)$.

 Writing $\alpha=(tx_1,\ldots,tx_n)\in \Aut_{\k[t,\frac{1}{t}]}$, the element
\[g=\alpha^{-1} \circ f\circ \alpha\] 
is an element of $\Aut_{\k[t]}(\A^n)$ that is such that replacing $t$ with $0$ in $g$ gives $g(0)=(Df)_p\in \SL_n(\k)$ (Lemma~\ref{AlexanderTrick}). 

Replacing $t$ with $t_0\in \k^*$ in $g$, we obtain a conjugate of $f$, that belongs then to $N$. As $N$ is closed, this implies that $g(0)=(Df)_p\in N$. As this latter is non-trivial elementary matrix, all such matrices are contained in $N$ and then $\SL_n(\k)\subseteq N.$
\end{proof}
We may now prove Proposition~\ref{NAut}:
\begin{proof}[Proof of Proposition~\ref{NAut}]
As $N$ is a closed normal subgroup of $\Aut_\k(\A^n)$, $N$ contains $\SL_n(\k)$ (Lemma~\ref{NclosedSLn}). The group $N\cap \SAut_\k(\A^n)$ is a  normal subgroup of $\SAut_\k(\A^n)$ that contains $\SL_n(\k)$, it then contains all translations (Proposition~\ref{SLnTame}).

We now show that $\GL_n(\k)$ is contained in $N$. For this, take $M\in \GL_n(\k)$, and choose $R\in N$ with $\Jac(R)=\det(M)$, which is possible since $\Jac(N)=\k^*$. Replacing $R$ with its composition with a translation, we may assume that $N$ fixes the origin.  We then do as in Lemma~\ref{AlexanderTrick}, write $\alpha=(tx_1,\ldots,tx_n)\in \Aut_{\k[t,\frac{1}{t}]}$ and obtain that
$g=\alpha^{-1}\circ  R \circ\alpha$
is an element of $\Aut_{\k[t]}(\A^n)$ that is such that $g(0)=(DR)_{(0,\ldots,0)}\in \GL_n(\k)$. Hence, $g(0)\in N$ and $\Jac(g(0))=\Jac(R)=\det(M)$, which implies that $M\cdot (g(0))^{-1}\in \SL_n(\k)$. Since $N$ contains $\SL_n(\k)$ and contains $g(0)$, it contains $M$.

We now show that $N=\Aut_\k(\A^n)$. For this we take $f\in \Aut_\k(\A^n)$ and prove that $f\in N$. As before, replacing $f$ with its composition with a translation, we may assume that $f$ fixes the origin. Choosing the same $\alpha$ as before, we obtain that
$g'=\alpha^{-1}\circ  f \circ\alpha$
is an element of $\Aut_{\k[t]}(\A^n)$ such that $g'(0)\in \GL_n(\k)$ (again by Lemma~\ref{AlexanderTrick}).

We then consider
\[ h=g'\circ f^{-1}\in \Aut_{\k[t]}(\A^n).\]
Replacing $t$ with $t_0\in \A^1\setminus \{0\}$ in $h$, we obtain $h(t_0)=g'(t_0)\circ f^{-1}= \alpha(t_0)^{-1}\circ  f \circ\alpha(t_0)\circ f^{-1}$, that belongs to $N$, as $\alpha(t_0)\in \GL_n(\k)$. Replacing $t$ with $0$ then also gives an elemen of $N$, that is $g'(0)\circ f^{-1}$. As $g'(0)\in \GL_n(\k)\subseteq N$, we find that $f\in N$, as desired.
\end{proof}
\bibliographystyle{alpha}
\bibliography{biblio}

\begin{thebibliography}{Kum02}

\bibitem[Bla16]{Blanc}
J{\'e}r{\'e}my Blanc.
\newblock Conjugacy classes of special automorphisms of the affine spaces.
\newblock {\em Algebra Number Theory}, 10(5):939--967, 2016.

\bibitem[Dan74]{Danilov}
V.~I. Danilov.
\newblock Nonsimplicity of the group of unimodular automorphisms of the affine
  plane.
\newblock {\em Math. Notes}, 15:165--167, 1974.

\bibitem[Edo18]{Edo}
Eric Edo.
\newblock Closed subgroups of the polynomial automorphism group containing the
  affine subgroup.
\newblock {\em Transform. Groups}, 23(1):71--74, 2018.

\bibitem[FK18]{FurterKraft}
Jean-Philippe Furter and Hanspeter Kraft.
\newblock On the geometry of the automorphism groups of affine varieties, 2018.

\bibitem[FL10]{FurterLamy}
Jean-Philippe Furter and St{\'e}phane Lamy.
\newblock Normal subgroup generated by a plane polynomial automorphism.
\newblock {\em Transform. Groups}, 15(3):577--610, 2010.

\bibitem[Jun42]{Jung}
H.~E.~W. Jung.
\newblock {\"U}ber ganze birationale {Transformationen} der {Ebene}.
\newblock {\em J. Reine Angew. Math.}, 184:161--174, 1942.

\bibitem[Kam96]{Kambayashi}
T.~Kambayashi.
\newblock Pro-affine algebras, ind-affine groups and the {Jacobian} problem.
  {Appendix}: {Linear} compactness for rings and modules.
\newblock {\em J. Algebra}, 185(2):481--498, 1996.

\bibitem[Kum02]{Kumar}
Shrawan Kumar.
\newblock {\em Kac-{Moody} groups, their flag varieties and representation
  theory}, volume 204 of {\em Prog. Math.}
\newblock Boston: Birkh{\"a}user, 2002.

\bibitem[Lew20]{Lewis}
D.~Lewis.
\newblock Normal subgroups generated by a single polynomial automorphism.
\newblock {\em Transform. Groups}, 25(1):177--189, 2020.

\bibitem[Mau01]{MaubachFinite}
Stefan Maubach.
\newblock Polynomial automorphisms over finite fields.
\newblock {\em Serdica Math. J.}, 27(4):343--350, 2001.

\bibitem[MP09]{MaubachPoloni}
Stefan Maubach and Pierre-Marie Poloni.
\newblock The {Nagata} automorphism is shifted linearizable.
\newblock {\em J. Algebra}, 321(3):879--889, 2009.

\bibitem[Nag72]{Nagata}
Masayoshi Nagata.
\newblock {\em On automorphism group of {{k[x,y]}}}, volume~5 of {\em Lect.
  Math., Dep. Math. Kyoto Univ.}
\newblock Tokyo: Kinokuniya Book-Store Co., Ltd. V, 1972.

\bibitem[Sha81]{Sha81}
I.~R. Shafarevich.
\newblock On some infinite-dimensional groups. {II}.
\newblock {\em Izv. Akad. Nauk SSSR, Ser. Mat.}, 45:214--226, 1981.

\bibitem[Sta12]{Stampfli}
Immanuel Stampfli.
\newblock On the topologies on ind-varieties and related irreducibility
  questions.
\newblock {\em J. Algebra}, 372:531--541, 2012.

\bibitem[SU04]{ShestakovUmirbaev}
Ivan~P. Shestakov and Ualbai~U. Umirbaev.
\newblock The tame and the wild automorphisms of polynomial rings in three
  variables.
\newblock {\em J. Am. Math. Soc.}, 17(1):197--227, 2004.

\bibitem[Zo11]{Zygadlo}
Jakub Zygad\l~o.
\newblock Remarks on a normal subgroup of {$GA_n$}.
\newblock {\em Comm. Algebra}, 39(6):1992--1996, 2011.

\end{thebibliography}

\end{document}